\documentclass[11pt]{article}

\usepackage{amsfonts}
\usepackage{amssymb}
\usepackage{amsmath}
\usepackage{amsthm}
\usepackage{color}
\usepackage{graphicx}

\newtheorem{theorem}{Theorem}

\newtheorem{corollary}[theorem]{Corollary}

\newtheorem{definition}[theorem]{Definition}

\newtheorem{lemma}[theorem]{Lemma}

\newtheorem{proposition}[theorem]{Proposition}
\newtheorem{remark}[theorem]{Remark}

\newtheorem*{thm2}{Theorem 2}
\newcommand{\comment}[1]{}
\renewcommand{\H}{\ensuremath{\mathbb{H}}}

\title{Dynamics on the $\textup{PSL}(2,\mathbb{C})$-character variety of a twisted $I$-bundle \thanks{Partially supported by NSF RTG grant DMS 0602191}}
\author{Michelle Lee}

\begin{document}

\maketitle

\begin{abstract}
   \centering
   \begin{minipage}{0.65\textwidth}
Let $M$ be a twisted interval bundle over a nonorientable hyperbolizable surface.  Let $\mathcal{X}(M)$ be the $\textup{PSL}(2, \mathbb{C})$-character variety of $\pi_1(M)$.  We examine the dynamics of the action of $Out(\pi_1(M))$ on $\mathcal{X}(M),$ and in particular, we find an open set on which the action is properly discontinuous that is strictly larger than the interior of the deformation space of marked hyperbolic manifolds homotopy equivalent to $M$.  Further, we identify which discrete and faithful representations can lie in a domain of discontinuity for the action of $Out(\pi_1(M))$ on $\mathcal{X}(M)$.
          \end{minipage}
\end{abstract}

\section{Introduction}
In this paper we use the deformation theory of hyperbolic $3$-manifolds to study the dynamics of $Out(\pi_1(M))$ on the $\textup{PSL}(2, \mathbb{C})$-character variety of $\pi_1(M)$ when $M$ is a hyperbolizable twisted $I$-bundle.  $AH(M)$ is the space of conjugacy classes of discrete and faithful representations of $\pi_1(M)$ into $\textup{PSL}(2, \mathbb{C})$.  It can also be thought of as the space of marked hyperbolic $3$-manifolds homotopy equivalent to a given compact $3$-manifold with boundary $M$.  It sits inside the $\textup{PSL}(2, \mathbb{C})$-character variety of $\pi_1(M)$ $$\mathcal{X}(M) = \textup{Hom}(\pi_1(M), \textup{PSL}(2, \mathbb{C})) /\!\!/ \textup{PSL}(2, \mathbb{C}).$$  $Out(\pi_1(M))$ acts on $\mathcal{X}(M)$ and on $AH(M)$ in the following way: an outer automorphism $[f]$ maps a representation $[\rho]$ to $[\rho \circ f^{-1}]$.  Using the parametrization of the interior of $AH(M)$ (see \cite{cm} for more details on this parametrization), it is well known that this action is properly discontinuous on the interior of $AH(M)$.  In this paper we find a larger domain of discontinuity for the action of $Out(\pi_1(M))$ in the case where $M$ is a twisted $I$-bundle over a non-orientable hyperbolic surface.  Namely, we prove the following.

\begin{theorem} \label{twistibund}
If $M$ is a hyperbolizable twisted $I$-bundle over a nonorientable hyperbolic surface, then there exists an open, $Out(\pi_1(M))$-invariant subset $\mathcal{PS}(M)$, called the set of primitive-stable representations, in $\mathcal{X}(M)$ containing the interior of $AH(M)$ as well as points on $\partial AH(M)$ such that the action of $Out(\pi_1(M))$ is properly discontinuous on $\mathcal{PS}(M)$.
\end{theorem}

Naturally, we would like to know whether $\mathcal{PS}(M)$ is a maximal domain of discontinuity.  Toward this end we show the following.

\begin{theorem} \label{maximality}
Let $[\rho]$ be an element of $AH(M)$.  Then $[\rho]$ lies in the complement of $\mathcal{PS}(M)$ if and only if there exists a primitive element $g$ of $\pi_1(M)$ such that $\rho(g)$ is parabolic.  Moreover, if $\rho$ lies in $AH(M) - \mathcal{PS}(M)$, then $\rho$ does not lie in any domain of discontinuity for the action of $Out(\pi_1(M))$ on $\mathcal{X}(M)$.
\end{theorem}

 A primitive element of $\pi_1(M)$ is one which is associated to a simple closed curve on the base surface. In particular, all purely hyperbolic points in $AH(M)$ lie in $\mathcal{PS}(M)$.

The notion of primitive-stable representations was first introduced by Minsky in \cite{min} where he proved a result analogous to Theorem \ref{twistibund} for handlebodies; namely, if $H_g$ is a genus $g$ hyperbolizable handlebody, then the set of primitive-stable representations, denoted $\mathcal{PS}(H_g)$, is a domain of discontinuity for the action of  $Out(F_g)$ strictly larger than the interior of $AH(H_g)$ where $F_g$ is the free group on $g$ generators.  In an upcoming paper (\cite{lee}), we extend the notion of primitive-stability to compression bodies.

Canary-Storm (\cite{cs}) showed that when $M$ is a compact hyperbolizable $3$-manifold with incompressible boundary and no toroidal boundary components that is not an $I$-bundle, then there exists an $Out(\pi_1(M))$-invariant open set $W(M)$ containing the interior of $AH(M)$ and containing points on $\partial AH(M)$ on which the action is properly discontinuous.  Canary-Storm use a different method of constructing this set that involves the characteristic submanifold of $M$ but the set they construct $W(M)$ also contains all purely hyperbolic points.  It is well known that in the case where $M$ is a trivial $I$-bundle over an orientable hyperbolic surface $S$ no point on the boundary of $AH(M)$ can lie in a domain of discontinuity for the action of $Out(\pi_1(M))$ on $\mathcal{X}(M)$ (see Section \ref{bpoints}).  Using these two results and Theorem \ref{twistibund} we can conclude the following.

\begin{corollary} \label{cor}
Let $M$ be a compact, orientable, hyperbolizable 3-manifold with nonempty incompressible boundary and no toroidal boundary components.  Then there exists an open, $Out(\pi_1(M))$-invariant set, containing the interior of $AH(M)$ and points on the boundary of $AH(M)$, on which $Out(\pi_1(M))$ acts properly discontinuously if and only if $M$ is not a trivial $I$-bundle over an orientable hyperbolic surface.  Moreover, in the case that $M$ is not a trivial $I$-bundle, then this set contains all purely hyperbolic points in $AH(M)$.
\end{corollary}

We remark that Corollary \ref{cor} suggests that the dynamics of a nonorientable surface group on its associated $\textup{PSL}(2, \mathbb{C})$-character variety differs significantly from that of an orientable surface group.  In the case of orientable surface groups, it is an open question whether the interior of deformation space is the largest domain of discontinuity for the action on its associated character variety.

We conclude this introduction with a brief outline of the paper.  In section \ref{prelims} we define the notion of primitive-stable representations and show that the collection of primitive-stable representations $\mathcal{PS}(M)$ forms an open $Out(\pi_1(M))$-invariant subset of $\mathcal{X}(M)$ on which $Out(\pi_1(M))$ acts properly discontinuously.  In section \ref{bpoints} we break up the proof of Theorem \ref{maximality} into two parts.  In Proposition \ref{characterizePS} we show that an element $[\rho]$ in $AH(M)$ lies in $\mathcal{PS}(M)$ if and only if $\rho(g)$ is not parabolic for any primitive element $g$ of $\pi_1(M)$.  Minsky observed that if a representation $[\sigma]$ maps the core curve of an essential annulus to a parabolic element, then $[\sigma]$ cannot lie in any domain of discontinuity.  Using this result and Proposition \ref{characterizePS} it follows that any point in $AH(M)$ outside of $\mathcal{PS}(M)$ does not lie in a domain of discontinuity, proving the second assertion of Theorem \ref{maximality}.  We conclude with the observation that Minsky's result can also be applied to show that no element in the boundary of quasi-fuchsian space can lie in a domain of discontinuity.

\vspace{0.2in}
\textbf{Acknowledgements: } I would like to thank my advisor Dick Canary for his constant support and guidance.  I would also like to thank Juan Souto and Nina White for several useful conversations. 

\section{Primitive stability} \label{prelims}
 The goal of this section is to show that the set $\mathcal{PS}(M)$ of primitive-stable representation is an open, $Out(\pi_1(M))$-invariant set on which the action is properly discontinuous.  Let $B$ be a nonorientable hyperbolic surface and let $\tilde B$ be its orientable double cover.  Let $M$ be a twisted $I$-bundle over $B$, namely $$M = \tilde B \times I / (x,t) \sim (\theta(x), 1-t)$$
 where $\theta$ is an orientation reversing, fixed-point free involution of $\tilde B$ such that $\tilde B / \langle \theta \rangle$ is homeomorphic to $B$.  Let $G$ be the fundamental group of $M$.

\begin{definition}
We say an element $g$ in $G$ is \emph{primitive} if it can be represented by a simple closed curve on the base surface $B$.
\end{definition}


As $G$ is torsion free, every element of $G$ acts with North-South dynamics on $C_S(G)$ the Cayley graph of $G$ with finite symmetric generating set $S$.
If $g_-$ and $g_+$ denote the repelling and attracting fixed points of $g$ on $\partial C_S(G)$, let $L(g)$ denote the set of geodesics connecting $g_-$ and $g_+$.  Let $\mathcal{P}$ denote the set of geodesics $l$ such that $l$ is contained in $L(g)$ for some primitive element $g$.

Given a representation $\rho: G \rightarrow \textup{PSL}_2(\mathbb{C})$ and a basepoint $x$ in $\H^3,$ there exists a unique $\rho$-equivariant map $\tau_{\rho, x}: C_S(G) \rightarrow \H^3$ taking  the identity to $x$ and edges to geodesic segments.

\begin{definition}
A representation $\rho: G \rightarrow \textup{PSL}_2(\mathbb{C})$ is called \emph{$(K, A)$-primitive-stable} if there exists a basepoint $x$ in $\H^3$ such that $\tau_{\rho, x}$ takes all geodesics of $\mathcal{P}$ to $(K, A)$-quasi-geodesics.
\end{definition}

Recall that a path $\alpha: \mathbb{R} \rightarrow \H^3$, parametrized by arclength, is a $(K, A)$-quasi-geodesic if for any $s$ and $t$ in $\mathbb{R}$ the following inequality holds.

$$
\frac{1}{K}|s-t|-A \leq d(\alpha(s), \alpha(t)).
$$

Primitive-stability is independent of the choice of basepoint in $\H^3$ and the choice of generators $S$, although the constants $(K, A)$ will change (see \cite{thesis} for details).

\begin{remark}
$(K, A)$-primitive-stability is invariant under conjugation by elements of $\textup{PSL}(2, \mathbb{C})$.  Two irreducible representations in $\textup{Hom}(G, \textup{PSL}(2, \mathbb{C}))$ lie in the same fiber of $\pi: \textup{Hom}(\pi_1(M), \textup{PSL}(2, \mathbb{C})) \rightarrow \mathcal{X}(M)$ if and only if they differ by conjugation (see \cite{hp}).  As reducible representations are never primitive-stable, primitive-stability is well defined on $\mathcal{X}(M)$.
\end{remark}

We will begin by showing that $\mathcal{PS}(M)$ is open and $Out(\pi_1(M))$-invariant.  Openness will follow from the fact that quasi-geodesics remain quasi-geodesics under small perturbations.  $Out(\pi_1(M))$-invariance will follow from the fact that automorphisms of $G$ preserve the set of primitive elements.
\begin{lemma} \label{basics}
$PS(M)$ is an open, $Out(\pi_1(M))$-invariant subset of $\mathcal{X}(M)$.  Moreover given any $[\rho_0]$ in $\mathcal{PS}$ there exists constants $(K_0, A_0)$ and a neighborhood $U_{[\rho_0]}$ such that any element $[\sigma]$ in $U_{[\rho_0]}$ is $(K_0, A_0)$-primitive-stable.
\end{lemma}

\begin{proof}
We start by showing that $\mathcal{PS}(M)$ is open.  The second statement will follow as a consequence of the proof of openness.
Let $L$ be a geodesic in $C_S(G)$, $L'$ the image of $L$ under $\tau_{\rho, x}$, $\{v_i\}$ the image of the vertex sequence of $L$, and $P_{j,i}$ the plane that perpendicularly bisects the geodesic segment $[v_{ji}, v_{(j+1)i}]$.
We will need the following characterization of quasi-geodesics due to Minsky (for proof see \cite{min} or \cite{thesis}).
\begin{lemma} [Minsky] \label{qg condition}
Given $(K, A)$, there exists $c > 0$ and $i \in \mathbb{N}$ such that if $L'=\tau_{\rho,x}(L)$ is a $(K, A)$-quasi-geodesic, then $P_{j,i}$ separates $P_{(j+1),i}$ and $P_{(j-1),i}$ and $d(P_{j,i}, P_{(j+1),i})>c$. Conversely, given $c>0$ and $i \in \mathbb{N}$ there exists $(K', A')$ such that if $L'=\tau_{\rho, x}(L)$ has the property that $P_{j,i}$ separates $P_{(j+1),i}$ and $P_{(j-1),i}$ and $d(P_{j,i}, P_{(j+1),i})>c$ then $L'$ is a $(K', A')$-quasi-geodesic.
\end{lemma}

To show that $\mathcal{PS}(M)$ is open, it suffices to show the following:

\begin{lemma} \label{qg on open sets}
For any primitive-stable representation $[\rho_0]$ in $\mathcal{X}(M)$, there exists $U_{[\rho_0]}$ a neighborhood of $[\rho_0]$ and constants $c'>0$, $i' \in \mathbb{N}$ such that for any $[\sigma]$ in $U_{[\rho_0]}$ and any geodesic $l$ in $\mathcal{P}$, the planes, $P_{j,i'}$ corresponding to $\tau_{\sigma, x}(l)$ have the property that $P_{j,i'}$ separates $P_{(j+1),i'}$ and $P_{(j-1),i'}$ and $d(P_{j,i'}, P_{(j+1),i'})>c'$.
\end{lemma}

We outline the proof of Lemma \ref{qg on open sets} whose details can be found in \cite{thesis}.  Suppose that $[\rho_0]$ is $(K_0, A_0)$-primitive-stable.  Let $i_0$ and $c_0$ be the constants obtained from Lemma \ref{qg condition}.  Let $i'=i_0$ and choose $c'<c_0$.  Since the map $\phi_{g}: \textup{Hom}(G, \textup{PSL}(2, \mathbb{C})) \rightarrow \H^3$ that sends $\rho \mapsto \rho(g)\cdot x$ is continuous and geodesics segments vary continuously with respect to their endpoints, for any pair of vertices $v_{ji_0}$ and $v_{(j+1)i_0}$ the point in the unit tangent bundle determining $P_{j,i_0}$ will vary continuously over $\textup{Hom}(G, PSL(2, \mathbb{C}))$.  As $G$ acts transitively and isometrically on $C_S(G)$ it suffices to check that the separation and distance properties of the planes $P_{j,i_0}$ hold for subpaths of elements of $\mathcal{P}$ beginning at the identity of length $3i_0$.  In particular we are only concerned with what happens to a finite number of group elements.  So if we take a lift $\rho_0$ of $[\rho_0]$ in $\textup{Hom}(G, PSL(2, \mathbb{C}))$ then there exists an open neighborhood $U_{\rho_0}$ of $\rho_0$ such that $P_{j,i'}$ separates $P_{(j+1),i'}$ and $P_{(j-1),i'}$ and $d(P_{j,i'}, P_{(j+1),i'})>c'$ whose image in $\mathcal{X}(M)$ is the desired neighborhood of $[\rho_0]$.  This completes the proof of Lemma \ref{qg on open sets} and hence openness.


To see that $\mathcal{PS}(M)$ is $Out(\pi_1(M))$-invariant first observe that since homotopy equivalences of closed surfaces are homotopic to homeomorphisms, any automorphism $f$ of $\pi_1(M) \cong \pi_1(B)$ preserves the set of primitive elements.  The image of $\tau_{\rho \circ f^{-1}, x}: C_{f(S)}(G) \rightarrow \H^3$ coincides with the image of $\tau_{\rho,x}: C_S(G) \rightarrow \H^3$.  Since primitive-stability is independent of the choice of generators of $G$, $\rho \circ f^{-1}$ is also primitive-stable.

\end{proof}

We finish this section by showing that the action of $Out(\pi_1(M))$ on $\mathcal{PS}(M)$ is properly discontinuous.  The idea is that primitive-stability will imply that translation length of a primitive element in the Cayley graph is coarsely the same as translation length of the corresponding isometry in $\H^3$.  To show proper discontinuity of the $Out(\pi_1(M))$-action it will suffice to show that only finitely many automorphisms, up to conjugation, can change the translation length of primitive elements in the Cayley graph by a bounded amount.
\begin{proposition}
$Out(\pi_1(M))$ acts properly discontinuously on $\mathcal{PS}(M)$.
\end{proposition}

\begin{proof}
Let $C$ be a compact subset of $\mathcal{PS}(M)$.  Let $l_\rho(g)$ denote the translation length of $\rho(g)$ in $\H^3$ and let $ || g ||$ denote the translation length of $g$ in $C_S(G)$.
We claim that there exist constants $r$ and $R$ such that

$$
r \leq \frac{l_\rho(g)}{||g||}\leq R
$$
for all primitive elements $g$ and all representations $[\rho]$ in $C$.

We will start by finding the upper bound.  Let $R$ be the maximum of $\{d(x, \rho(s)x\}$ over $s$ in $S$ and $[\rho]$ in $C$; such a $R$ exists as $S$ is finite and the map $\mathcal{X}(M) \rightarrow \mathbb{R}$ that maps $[\rho]$ to $d(x, \rho(s)x)$ for a fixed $s$ and $x$ is continuous.  Then, $\tau_{\rho, x}$ is $R$-Lipschitz for all $[\rho]$ in $C$ and so
$$
l_\rho(g) \leq R ||g||.
$$

To see the lower bound we can first assume by Lemmas \ref{qg condition} and \ref{qg on open sets} that there exists $(K, A)$ such that every element $[\rho]$ in $C$ is $(K, A)$-primitive-stable.  We would like to compare translation length of $g$ along the geodesics in $L(g)$ with $l_\rho(g)$.  Unfortunately, $g$ need not act by translation on the elements of $L(g)$.  Instead, there exists uniform quasi-axes for $g$ meaning there exists (a not necessarily unique) $(K', A')$-quasi-geodesic $l'$ invariant under the action of $g$ where $K'$ and $A'$ do not depend on $g$.  Moreover, the translation length along these quasi-axes is uniformly bounded in $||g||$; namely, there exist constants $\lambda,c$ independent of $g$ such that $d(z, gz) \leq \lambda||g||+c$ for any $z$ in $l'$ (see \cite{thesis}).
It will suffice to consider the image of these quasi-axes by the following lemma (see \cite{thesis} for proof).
\begin{lemma} \label{quasigoesquasi}
Given $(K, A)$ and $(K', A')$, there exists $K''=K''(K,K',A,A')$ and $A''=A''(K,K',A,A')$ such that the following holds.  If $\gamma$ is a geodesic in $C_S(G)$ such that $\tau_{\rho, x}(\gamma)$ is a $(K, A)$-quasi-geodesic and $\gamma'$ is a $(K', A')$-quasi-geodesic in $C_S(G)$ with the same endpoints at infinity as $\gamma$, then $\tau_{\rho,x}(\gamma')$ is a $(K'', A'')$-quasi-geodesic.
\end{lemma}

Using Lemma \ref{quasigoesquasi} we can see that if $l'$ is a $(K', A')$-quasi-axis for $g$ then $\tau_{\rho,x}(l')$ is a $(K'', A'')$-quasi-geodesic in $\H^3$.  In particular, $\tau_{\rho, x}(l')$ lies in a $R''=R''(K'',A'')$ neighborhood of the axis for $\rho(g)$.
Hence if we take $y$ to lie on the image of a quasi-axis for $g$, we have

$$
l_\rho(g) \geq d(y, \rho(g) \cdot y) -2R'' \geq \frac{||g||}{K''}-A''-2R''.
$$

So,
$$ \frac{l_\rho(g)}{||g||} \geq \frac{1}{K''}-\frac{A''}{||g||}-\frac{2R''}{||g||}.$$

For $||g||$ larger than $2(A''-2R'')K''$,
$$
\frac{1}{K''}-\frac{A''}{||g||}-\frac{2R''}{||g||} > \frac{1}{2K''} >0.
$$

Let $[g_1], \ldots, [g_m]$ be the conjugacy classes of elements of $G$ such that $||g_i||$ is less than $2(A''-2R'')K''$.  Since $\frac{l_\rho(g_i)}{||g_i||}$ varies continuously over $\mathcal{X}(M)$ for each $g_i$ there exists a minimum value $r_i$ for $\frac{l_\rho(g_i)}{||g_i||}$ over $C$.  Take $r$ to be the minimum of $\{r_1, \ldots, r_m, \frac{1}{2K''}\}$.

Now suppose that $[f]$ is an element in $Out(G)$ such that $f \cdot C \cap C \neq \emptyset$.  Applying the above inequalities we have that for any $[\rho] $ in $C$ and any primitive word $w$,

$$
||f^{-1}(w)|| \leq \frac{1}{r} l_\rho(f^{-1}(w))= \frac{1}{r} l_{\rho \circ f^{-1}}(w) \leq  \frac{R}{r}||w||.
$$

Let $x_1, \ldots, x_n$ be a set of generators for $G$ such that each $x_i$ is primitive and $x_ix_j$ where $i \neq j$ is primitive; the standard generators for $\pi_1(B)$ will do.  Let $\mathcal{W}=\{x_i, x_ix_j |  \text{ } i \neq j\}$.  Then it suffices to show the following lemma.

\begin{lemma}
For any $N > 0$, the set $$\mathcal{A}=\{[f] \in Out(G) | \text{    } ||f(w)|| \leq N ||w|| | \text{ for all $w \in \mathcal{W}$} \}$$ is finite.
\end{lemma}

\begin{proof}
Suppose that $\{[f_k]\}$ is a sequence of infinitely many distinct elements in $\mathcal{A}$.  As $G$ acts cocompactly on $\H^2$, there is a $G$-equivariant quasi-isometry $\tau': C_S(G) \rightarrow \H^2$.  Fix the following notation: Let $\overline{g}$ denote the isometry of $\H^2$ induced by the action of $g$, $l(\overline{g})$ its translation length and $\textup{Ax}(\overline{g})$ its invariant geodesic axis.  As $G$ acts cocompactly on $\H^2$, there exists $r>0$ such that $l(\overline{f_k(x_i)}) \geq r$ on $\H^2$ for all $i$ and $k$.  Since $||f_k(x_ix_j)|| \leq 2N$, there exists $R$ such that $l(\overline{f_k(x_ix_j)}) \leq R$ all $k$ and all pairs $i,j$ such that $i \neq j$.

This implies that there exists an upper bound $D$ on the distance between $\textup{Ax}(\overline{f_k(x_i)})$ and $\textup{Ax}(\overline{f_k(x_j)})$ for all $k$ and all pairs $i,j$.  If, to the contrary, $\{d(\textup{Ax}(\overline{f_k(x_i)}), \textup{Ax}(\overline{f_k(x_j)})\}$ was unbounded, since $l(\overline{f_k(x_i)})$ and $l(\overline{f_k(x_j)})$ are bounded from below by $r$, $\{l(\overline{f_k(x_i)f_k(x_j)})\}$ would also be unbounded, a contradiction.

 Then, there also exists an upper bound $D'$ on the distance between $\textup{Ax}(f_k(x_i))$, a quasi-axis of $f_k(x_i)$, and $\textup{Ax}(f_k(x_j))$, a quasi-axis of $f_k(x_j)$, for all $i,j$ and $k$.  Up to conjugation, we can assume that $\textup{Ax}(f_k(x_i))$ is a uniformly bounded distance $D''$ from the identity $e$ for all $k$ and $i$.  If $y$ is a point on $\textup{Ax}(f_k(x_i))$ closest to $e$, then we can bound the distance between the identity and $f_k(x_i)$ in the Cayley graph as follows.  Recall that there exists constants $\lambda$ and $c$ such that the translation length of $g$ along its quasi-axis $\textup{Ax}(g)$ is at most $\lambda ||g|| + c$, such that $\lambda$ and $c$ are independent of $g$.
 \begin{eqnarray*}
 d(e, f_k(x_i)) &\leq& D'' + d(y, f_k(x_i)y) + D'' \\
 &\leq & 2 D'' + \lambda ||f_k(x_i)|| + c \\
 & \leq & 2 D'' + \lambda N + c
 \end{eqnarray*}
 This implies that up to conjugation, there are only finitely many possibilities for $f_k(x_i)$.  Hence $\mathcal{A}$ must be finite.

\end{proof}
\end{proof}
To complete the proof of Theorem \ref{twistibund} it remains to show that $\mathcal{PS}(M)$ is strictly larger than the interior of $AH(M)$.  This will follow immediately from Theorem \ref{maximality}, which will be proven in section \ref{bpoints}.

\section{Primitive-stable points on the boundary of $AH(M)$} \label{bpoints}
In this section we prove Theorem \ref{maximality}.  We break up the proof into Propositions \ref{characterizePS} and \ref{notdd}.  We start by characterizing which points in $AH(M)$ lie in $\mathcal{PS}(M)$.  The interior of $AH(M)$ consists of convex cocompact representations (see \cite{sul}), namely those representations whose associated hyperbolic manifold has a compact convex core.  We will show that $[\rho]$ in $AH(M)$ lies in $\mathcal{PS}(M)$ if and only if $\rho(g)$ is hyperbolic for all primitive elements $g$ of $G$.  In particular, $\mathcal{PS}(M)$ will contain the interior of $AH(M)$ as well as all purely hyperbolic points on the boundary of $AH(M)$.  This will complete the proof of Theorem \ref{twistibund}.  Then, to complete the proof of Theorem \ref{maximality} we will use an observation by Minsky that if $[\sigma]$ in $AH(M')$ maps a core curve of an essential annulus to a parabolic element of $\textup{PSL}(2, \mathbb{C})$ then $[\sigma]$ cannot lie in a domain of discontinuity of the action of $Out(\pi_1(M'))$ on $\mathcal{X}(M')$.  Finally we conclude with the result that no point on the boundary of quasi-Fuchsian space can lie in a domain of discontinuity.

We start this section by reviewing some basic facts from hyperbolic geometry that we will need.  There exists a constant $\mu_3>0$ such that for any hyperbolic $3$-manifold $N \cong \H^3/ \Gamma$, $\Gamma$ a discrete subgroup of $Isom^+(\H^3)$ and any $\epsilon < \mu_3$ each component of $N_{thin(\epsilon)}=\{x \in N | inj_N(x) < \epsilon\}$ is either a metric neighborhood of a closed geodesic or a parabolic cusp homeomorphic to either $S^1 \times \mathbb{R} \times (0, \infty)$ or to $T \times (0, \infty)$ where $T$ is a torus (see \cite{bp}, Chp. D).  $N^0_\epsilon$ denotes the complement of the non-compact portions of $N_{thin(\epsilon)}$ in $N$.  The convex core $C(N)$ of $N$ is the smallest convex submanifold of $N$ such that the inclusion of $C(N)$ into $N$ is a homotopy equivalence.  $N$ is called convex cocompact if $C(N)$ is compact.  In general, when $\pi_1(N)$ is finitely generated, there exists a compact submanifold $C$, called the compact core, whose inclusion induces a homotopy equivalence with $N$ (see \cite{scott}).   Moreover, $C$ can be chosen such that $C$ intersects each component of the noncompact portions of $N_{thin(\epsilon)}$ in a single incompressible annulus or torus (see \cite{mcc}).  A compact core of the latter type is called a relative compact core.

The goal of this section is to prove the following.

\begin{thm2}
Let $[\rho]$ be an element of $AH(M)$.  Then $[\rho]$ does not lie in $\mathcal{PS}(M)$ if and only if there exists a primitive element $g$ of $\pi_1(M)$ such that $\rho(g)$ is parabolic.  Moreover, if $\rho$ lies in $AH(M) - \mathcal{PS}(M)$, then $\rho$ does not lie in any domain of discontinuity for the action of $Out(\pi_1(M))$ on $\mathcal{X}(M)$.
\end{thm2}

We will use the following characterization of which discrete and faithful representations mapping all primitive elements to hyperbolic elements of $\textup{PSL}(2, \mathbb{C})$ are primitive-stable.
\begin{lemma} \label{compact}
Let $\rho$ be a discrete and faithful representation of $\pi_1(M)$ into $\textup{PSL}(2, \mathbb{C})$ such that $\rho(g)$ is hyperbolic for any primitive element $g$.  Then $\rho$ is primitive-stable if and only if there exists a compact subset $\Omega$ of $N_\rho = \H^3/\rho(\pi_1(M))$ such that the set of geodesics corresponding to primitive elements of $\pi_1(M)$ is contained in $\Omega$.
\end{lemma}

\begin{proof}
Suppose there exists a compact set $\Omega$ such that all primitive geodesics of $N_\rho$ are contained in $\Omega$.  Without loss of generality we can assume that $\Omega$ is a compact core $C$ of $N_\rho$ containing the image of $C_S(G)/ \rho(G)$ in $N_\rho$.  This implies, in particular, that $\tilde \Omega$, the preimage of $\Omega$ in $\H^3$ is connected.  For some $(K, A)$, $\tau_{\rho, x}:C_S(G) \rightarrow \tilde \Omega \subset \H^3$ is a $(K, A)$-quasi-isometry from $C_S(G)$ to $\tilde \Omega$ with the intrinsic metric.  Any geodesic $l$ in $\mathcal{P}$ connecting $g_-$ and $g_+$, the fixed points of $g$, maps to a $(K, A)$-quasi-geodesic in $\tilde \Omega$, with the intrinsic metric.  In particular, $\tau_{\rho,x}(l)$ lies in a $R=R_\Omega(K, A)$-neighborhood of $\textup{Ax}(g)$, a lift of the geodesic representing $g$ in $N_\rho$.  Then $\tau_{\rho,x}(l)$ lies in a $R$ neighborhood of $\textup{Ax}(g)$ with the extrinsic metric on $\tilde \Omega$.  If $x,y$ lie on $\tau_{\rho, x}(l)$ and if $\pi$ denotes the closest point projection onto $\textup{Ax}(g)$ in $\tilde \Omega$, then
$$
d_{\tilde \Omega}(x, y) \leq d_{\tilde \Omega}(\pi(x),\pi( y)) + 2R = d_{\H^3}(\pi(x), \pi(y)) + 2R \leq d_{\H^3} (x, y)+4R
$$
This implies that $\tau_{\rho, x}(l)$ is a $(K, A+4R)$-quasi-geodesic in $\tilde \Omega$ with the extrinsic metric.  Hence $\rho$ is $(K, A+4R)$-primitive-stable.
Conversely, if $\rho$ is $(K, A)$-primitive-stable then elements of $\mathcal{P}$ stay within a bounded neighborhood of their corresponding geodesic axes in $\H^3$.  In particular, geodesics representing primitive elements will stay in a bounded neighborhood of the image of the Cayley graph in $N_\rho$, which is a compact set.
\end{proof}

We will start by proving the first assertion of Theorem \ref{maximality}.

\begin{proposition} \label{characterizePS}
Let $[\rho]$ be an element of $AH(M)$.  Then $[\rho]$ does not lie in $\mathcal{PS}(M)$ if and only if there exists a primitive element $g$ of $\pi_1(M)$ such that $\rho(g)$ is parabolic.
\end{proposition}

\begin{proof}
The backwards direction is easy for if $\rho(g)$ is parabolic, then for any geodesic $l$ connecting the fixed points $g_+$ and $g_-$ on $\partial C_S(G)$, $\tau_{\rho, x}(l)$ is not quasi-geodesic.

For the forward direction, if $\rho(g)$ is hyperbolic for every primitive $g$, then by Lemma \ref{compact} it suffices to check that closed geodesics corresponding to primitive elements remain in a compact set.  Let $\gamma_g$ denote the unique geodesic representative of $\rho(g)$ in $N$.  The representation $\rho$ induces a homotopy equivalence $h_\rho:M \rightarrow N$.  Precompose with the inclusion $B \rightarrow M$, to obtain an incompressible map $h_\rho': B \rightarrow N$.  Let $\alpha_g$ be a simple closed curve on $B$ such that $h_\rho'(\alpha_g)$ is freely homotopic to $\gamma_g$.  Fix a point $x_0$ on $\alpha_g$.  Extend $\alpha_g$ and $x_0$ to a one vertex triangulation of $B$, meaning a collection of mutually nonisotopic arcs $k_i$ with all endpoints at $x_0$ and disjoint interiors such that $B- \cup k_i \cup \alpha_g$ is a collection of triangles.

The map $h_\rho'$ is homotopic to a map $h_g$ such that $\alpha_g$ maps to $\gamma_g$, each arc $k_i$ is mapped to a geodesic arc and each triangle is mapped to a totally geodesic triangle.
If we endow $B$ with the pull-back metric, then $B$ has a hyperbolic metric with one cone singularity at $x_0$.  By construction, the sum of the angles of the sectors around $x_0$ is at least $2\pi$. In particular, the area of $B$ is bounded above by $-2\pi\mathcal{X}(B)$.  We have constructed a so-called simplicial hyperbolic surface (for more details see \cite{bon} \S 1.2).

We claim that given any such simplicial hyperbolic surface, $h: B \rightarrow N$, there is a uniform upper bound on how far its image can venture out of $N^0_\epsilon$.  Observe first that there exists a constant $A$ depending only on the Euler characteristic of $B$ such that for any point $x$ in $B$ there exists a homotopically nontrivial simple curve through $x$ of length less than $A$.  To produce such a curve take a ball centered at $x$ and blow it up until it intersects itself.  Since the area of $B$ is bounded, there is a uniform upper bound on the area and radius of such a ball.

For any $L> 2A$, if $h(B) \cap (\mathcal{N}_L(N^0_\epsilon)- N^0_\epsilon) \neq \emptyset$ then there exists a homotopically non-trivial simple curve entirely contained within a noncompact component of $N_{thin(\epsilon)}$.  This implies that the curve represents a parabolic element, a contradiction.  So, there exists $\epsilon_0$ such that simplicial hyperbolic surfaces realizing primitive geodesics are contained in $N^0_{\epsilon_0}$  

Now suppose to the contrary that $\{\gamma_i\}$ is a sequence of primitive geodesics not contained in any compact set.  Let $h_i: B \rightarrow N$ denote a simplicial hyperbolic surface containing $\gamma_i$.  We can lift $h_i$ to a map $\tilde h_i: S \rightarrow \tilde N$ where $S = \partial M$ and $\tilde N$ is the double cover of $N$ associated to the subgroup $\pi_1(S)$.  The map $\tilde h_i$ is a simplicial hyperbolic surface containing $\gamma_i$ and $\tilde \theta(\gamma_i)$ where the associated triangulation is the preimage of the triangulation on $B$.  Moreover, by construction, $\tilde h_i$ satisfies $\tilde \theta \circ \tilde h_i = \tilde h_i \circ \theta$, where $\tilde \theta$ is the nontrivial covering transformation of $\tilde N$.

Fix $C$ a compact core for $N$. The preimage $\tilde C$ of $C$ in $\tilde N$ is a compact core for $\tilde N$.  As $\tilde C$ is homotopy equivalent to a fiber surface $S$ which separates $\tilde N \cong S \times \mathbb{R},$ $\tilde N - \tilde C$ has two components.  Since $\tilde C$ covers $C$ and $C$ has only one boundary component, $\tilde \theta$ must exchange the two boundary components of $\tilde C$, and hence $\tilde \theta$ must exchange the two components of $\tilde N - \tilde C$.

As $\{\gamma_i\}$ is not contained in any compact set, we can assume, up to subsequence, that there exists a point $x_i$ on $\gamma_i$ such that $x_i$ lies outside the compact set $C_i$ where $C_i$ is defined as
$$C_i = \{x \in N^0_{\epsilon_0} \text{  }| \text{  there exists a path $c$ from $x$ to $C$ such that } l( c  \cap N_{thick(\epsilon_0)} ) \leq i\}.$$

As $\tilde \theta$ interchanges the two components of $\tilde N - \tilde C$, the two lifts $\tilde x_i, \tilde x_i'$ of $x_i$ lie in different components of $\tilde N - \tilde C$, but by the equivariance property of $\tilde h_i$, they both lie on $\tilde h_i(S)$.  Any path $c$ on $\tilde h_i(S)$ connecting $\tilde x_i$ and $\tilde x_i'$ satisfies $d(c \cap \tilde N_{thick(\epsilon_0)}) \geq 2i$, for if not, in $N$ there would be a path $c'$ connecting $x_i$ to $C$ with $d(c' \cap N_{thick(\epsilon_0)}) < i$.  For $i$ large enough, this contradicts Bonahon's bounded diameter lemma (\cite{bon} Lemma 1.11) that states that the diameter of any incompressible simplicial hyperbolic surface modulo the $\epsilon_0$-thin part is bounded above where the bound depends only on $\epsilon_0$ and the topology of the surface.

\end{proof}

This completes the proof of the first assertion of Theorem \ref{maximality}.  To see the second assertion we will need the following observation due to Minsky.

\begin{lemma} [Minsky] \label{infstab}
Let $M$ be a compact hyperbolizable manifold with no toroidal boundary components.  Let $\gamma$ be the core curve of an essential annulus in $M$.  Suppose that $\rho: \pi_1(M) \rightarrow \textup{PSL}_2(\mathbb{C})$ is a discrete and faithful representation such that $\rho(\gamma)$ is parabolic.  Then any neighborhood of $[\rho_0]$ contains points with infinite stabilizers.  In particular, $[\rho]$ cannot lie in a domain of discontinuity for the action of $Out(\pi_1(M))$ on $\mathcal{X}(M)$.
\end{lemma}

\begin{proof}
 Consider the map $$\textup{tr}^2_\gamma: \mathcal{X}(M) \rightarrow \mathbb{C}$$ where $[\rho] \mapsto \textup{tr}([\rho(\gamma)])^2$.  As a neighborhood of $AH(M)$ is a smooth complex manifold (see \cite{kap} Chapter 4), on which $\textup{tr}^2_\gamma$ is a holomorphic map, $\textup{tr}^2$ is either constant or open on that neighborhood.  As the interior of $AH(M)$ consists of convex cocompact representations that are dense in $AH(M)$ (see \cite{bcm2} and \cite{ns}), the image of $\textup{tr}^2_\gamma$ cannot be constant on all of $AH(M)$ and hence it must be an open map on that neighborhood of $AH(M)$.  Since isometries of $\H^3$ are determined, up to conjugacy, by their trace and there are finite order elliptic isometries with trace arbitrarily close to $2$ or $-2$, there exist representations $\rho_i$ approaching $\rho$ such that $\rho_i(\gamma)$ corresponds to a finite order elliptic isometry.  Let $n_i$ denote the order of $\rho_i(\gamma)$.  Then $D^{n_i}_\gamma$ the Dehn twist of order $n_i$ about the annulus whose core curve is $\gamma$ is an element in $Out(\pi_1(M))$ that fixes $[\rho_i]$.  Hence, elements arbitrarily close to $[\rho]$ have infinite stabilizers.
\end{proof}

\begin{proposition} \label{notdd}

If $\rho$ lies in $AH(M) - \mathcal{PS}(M)$, then $\rho$ does not lie in any domain of discontinuity for the action of $Out(\pi_1(M))$ on $\mathcal{X}(M)$.
\end{proposition}

\begin{proof}
If $[\rho]$ lies in the complement of $\mathcal{PS}(M)$ in $AH(M)$, then there exists a primitive element $g$ such that $\rho(g)$ is parabolic.  Then $g$ is either the core curve of an essential annulus or the core curve of an essential Mobius band.  By Lemma \ref{infstab} it suffices to show that the latter is impossible.   Suppose that $\gamma$ is a closed essential curve in $M$ mapping to a parabolic element in $N_\rho=\H^3/\rho(G)$.  Then we claim that $\gamma$ must be homotopic into $\partial M$.  If $C$ is a relative compact core for $N_\rho$, consider the map $\phi: S=\partial M \rightarrow C$ in the homotopy class of $\rho|_{\pi_1(\partial M)}$.  If $\tilde C$ is the cover of $C$ associated to the subgroup $\phi_*(\pi_1(S))$ it is a compact manifold with $\pi_1(\tilde C) \cong \pi_1(S)$.  This implies that $\tilde C$ must be a trivial $I$-bundle over $S$ (\cite{hem} Theorem 10.6).  Then the lift $S \rightarrow \tilde C$ can be homotoped into $\partial \tilde C$.  Hence the map $S \rightarrow C$ is also homotopic into the boundary of $C$.  Then, we can homotope the map $M \rightarrow C$ to a map that sends $\partial M $ into $\partial C$.  This map is either homotopic to a homeomorphism or is homotopic to a map $M \rightarrow \partial C$ (\cite{hem} Theorem 13.6).  The latter cannot happen as this would imply that $\partial C$ is a nonorientable closed surface. Hence any curve in $M$ mapping to a parabolic element in $N_\rho$ is homotopic into $\partial M$.  As the core curve of an essential Mobius band cannot be homotoped into $\partial M$, it cannot be mapped to a parabolic element in $N_\rho$.
\end{proof}

This completes the proof of Theorem \ref{maximality}.  We end this section with an application of Minsky's observation in the quasi-Fuchsian case.
\begin{proposition} \label{qfs}
Let $F$ be an orientable hyperbolic surface.  Then no point on the boundary of $AH( F \times I)$ can lie in a domain of discontinuity for the action of $Out(\pi_1(F \times I))$ on $\mathcal{X}(F \times I)$.
\end{proposition}

\begin{proof}
Since geometrically finite points are dense on the boundary of $AH(F \times I)$ (see \cite{bcm1} and \cite{ch}) and since any simple closed curve on $F$ is the core curve of an essential annulus, the result follows from Lemma \ref{infstab}.
\end{proof}

\end{document}